\documentclass[12pt]{amsart}

\usepackage{tikz}
\usetikzlibrary{snakes}
\usepackage[colorlinks=true, linkcolor=blue, anchorcolor=blue, citecolor=blue, filecolor=blue, menucolor= blue, urlcolor=blue]{hyperref}
\usepackage{amsmath,amssymb,amsthm,amscd}
\usepackage{verbatim}
\usepackage{comment}
\usepackage{multirow}
\usepackage{mathtools}
\usepackage{enumitem}
\usepackage{pgf,tikz}
\usetikzlibrary{arrows}
\usepackage[normalem]{ulem}

\newcommand{\ffi}{\varphi}

\usepackage[margin=1in]{geometry} 

\numberwithin{equation}{section}

\newtheorem{theorem}{Theorem}[section]
\newtheorem{proposition}[theorem]{Proposition}
\newtheorem{lemma}[theorem]{Lemma}
\newtheorem{corollary}[theorem]{Corollary}

\newtheorem{conjecture}[theorem]{Conjecture}
\newtheorem*{theorem*}{Theorem}

\theoremstyle{definition}

\newtheorem{notation}[theorem]{Notation}

\newtheorem{remark}[theorem]{Remark}

\newcommand{\NN}{ \ensuremath{\mathbb{N}}}

\newcommand{\PP}{ \ensuremath{\mathbb{P}}}

\newcommand{\reg}{\ensuremath{\mathrm{reg}}\hspace{1pt}}

\definecolor{MyDarkGreen}{cmyk}{0.7,0,1,0}

\def\cocoa{{\hbox{\rm C\kern-.13em o\kern-.07em C\kern-.13em o\kern-.15em A}}}

\begin{document}

\title[Powers of linear forms]{The Lefschetz question for ideals generated by powers of  linear forms in few variables}

\author{J.\ Migliore} 
\address{Department of Mathematics \\
University of Notre Dame \\
Notre Dame, IN 46556 USA}
 \email{migliore.1@nd.edu}

\author{U.\ Nagel}
\address{Department of Mathematics\\
University of Kentucky\\
715 Patterson Office Tower\\
Lexington, KY 40506-0027 USA}
\email{uwe.nagel@uky.edu}

\begin{abstract} 
The Lefschetz question asks if multiplication by a power of a general linear form, $L$, on a graded algebra has maximal rank (in every degree). We consider a quotient by an ideal that is generated by powers of linear forms. Then the Lefschetz  question is, for example, related to the problem whether a set of fat points imposes the expected number of conditions on a linear system of hypersurfaces of fixed degree. Our starting point is a result that relates Lefschetz properties in different rings. It suggests to use induction on the number of variables, $n$. If $n = 3$, then it is known that  multiplication by $L$ always has maximal rank. We show that the same is true for multiplication by $L^2$ if all linear forms are general. Furthermore, we give a complete description of when multiplication by $L^3$ has maximal rank (and its failure when it does not). 
As a consequence, for such ideals that contain a quadratic or cubic generator, we establish results on the so-called Strong Lefschetz Property for ideals in $n=3$ variables, and the Weak Lefschetz Property for ideals in $n=4$ variables.
\end{abstract}


\thanks{
{\bf Acknowledgements}: 
Migliore was partially supported by Simons Foundation grant \#309556.
Nagel was partially supported by Simons Foundation grant \#317096. \\
This paper originated from work done during a visit of Migliore to the University of Kentucky (see \cite{MN}), and he thanks the Department of Mathematics for their kind hospitality. The authors are greatly indebted to Craig Huneke for fruitful discussions and for allowing us to include their outcomes here. These discussions took place at a workshop in Oaxaca, and the authors thank its organizers, Chris Francisco, Tai Ha and Adam van Tuyl, for bringing us  together and for the stimulating atmosphere.  
Both authors also thank Brian Harbourne for useful comments.
}



\maketitle



\section{Introduction}

The Lefschetz questions in commutative algebra refer to the study of the rank of the homomorphism induced by multiplication by powers $L^k$ of a general linear form $L$ on components of some graded artinian quotient, $R/I$, of a polynomial ring $R = K[x_1,\dots,x_r]$ over an infinite field $K$. In particular, when $k=1$, if this homomorphism always has maximal rank (regardless of the  component chosen as the domain of the homomorphism) then we say that $R/I$ has the {\em Weak Lefschetz Property (WLP)} and if it always has maximal rank also for all choices of the exponent $k$ then $R/I$ is said to have the {\em Strong Lefschetz Property (SLP)}. 

Different families of ideals exhibit very different kinds of behavior. For monomial complete intersections in characteristic zero, it was shown by Stanley \cite{stanley} and later by Watanabe \cite{watanabe} that the SLP always holds. Stanley \cite{St-faces} used the WLP of certain toric ideals to prove one direction of the celebrated $g$-theorem. These results generated many natural questions, and many at least partial answers. For instance, one can ask if all complete intersections share some Lefschetz property (in characteristic zero), ideally SLP. There is no known counterexample, but the best result to date \cite{HMNW} is that all complete intersections in three variables have the WLP. More generally, it is known that not all Gorenstein quotients of $R$ have the WLP, but it is open whether in three variables they all have the SLP, or even the WLP \cite{BMMNZ2}. 

The Stanley-Watanabe result also led to a careful study of which monomial ideals have the WLP (especially), including a careful analysis of the role of the characteristic in these questions  \cite{BMMNZ1}, \cite{BK}, \cite{CMNZ},  \cite{cook-nagel 1}, \cite{cook-nagel 2},  \cite{kustin-vraciu}, \cite{LZ}, \cite{MMN1}. However, in this paper we always assume that the base field $K$ has characteristic zero.

Finally, the Stanley-Watanabe result led to the natural question of what can be said for an ideal generated by powers of more than $r$ linear forms, and by now a number of papers have contributed to a growing understanding of this situation. In \cite{SS}, the authors made the striking observation that in three variables (i.e. $r=3$), {\em any} such ideal has the WLP. It was shown in \cite{DIV} and \cite{CHMN} that  for multiplication by $L^2$, the same is not true even in three variables, by making a careful choice of the linear forms. Thus several papers studied these questions for ideals generated by powers of {\em general} linear forms, both in $K[x,y,z]$ and in rings with more variables. These include \cite{SS}, \cite{MMN1}, \cite{HSS}, \cite{MMN2}, \cite{M},  \cite{MM}. More often than not, the result was that not even WLP  holds, but it is fascinating to see situations where maximal rank does hold. In fact, this is related to failure of a fat point scheme to impose the expected number of conditions on hypersurfaces of some degree (see, e.g., \cite{MMN2, DIV, CHMN}). 

Complementing the above-mentioned results of \cite{SS} on the one hand and  \cite{DIV} and \cite{CHMN} on the other,  in \cite{MM} the authors began the study of multiplication by low-degree powers of a general linear form. It was shown that if $I  = (L_1^t,\dots,L_s^t) \subset K[x,y,z]$ is generated by {\em uniform} $t$-th powers of general linear forms, then $\times L^2 : [R/I]_{j-2} \rightarrow [R/I]_j$  has maximal rank, for all $j$. It was also shown there that for ideals generated by {\em four} uniform $t$-th powers of general linear forms, $\times L^3$ has maximal rank two-thirds of the time (see Theorem \ref{MM 5.1} below for the precise statement), $\times L^4$ has maximal rank one-third of the time, and $\times L^5$ fails maximal rank in exactly one degree for all $t \geq 4$.   It was also shown in \cite{MMN2} that if $I = (L_1^{a_1},\dots, L_4^{a_4}) \subset K[x,y,z]$ then $\times L^2 : [R/I]_{j-2} \rightarrow [R/I]_j$  has maximal rank for all $j$, regardless of the values of $a_1,\dots,a_4$.

The starting point of this paper is a general exchange result that relates Lefschetz properties of different rings (see Proposition~\ref{lem:exchange}). 
It suggests an inductive approach based on the number of variables and/or the number of minimal generators. In fact, a special case of Proposition~\ref{lem:exchange} (b) combined with the fact that every artinian ideal in a polynomial ring of at most two variables has the SLP and every complete intersection in three variables has the WLP (both from \cite{HMNW}) implies the mentioned WLP of ideals generated by powers of linear forms in three variables (see also \cite{MMN2}). 

The WLP says that multiplication by a general linear form $L$ on the quotient ring has  maximal rank in every degree. Continuing in the direction begun in \cite{MMN2} and especially in \cite{MM}, here we establish a complete description of the maximal rank property of $\times L^2$ and $\times L^3$ for ideals generated by arbitrary many powers of general linear forms. 
To this end we give a new approach to part of the argument, which significantly simplifies the computations. Our main result in this regard (from Theorems \ref{xL2}, \ref{xL3} and \ref{xL3 general}) is the following.


\begin{theorem} 
   \label{thm:intro1}
In $R = K[x,y,z]$ consider the ideal $I = (L_1^{a_1},\dots,L_s^{a_s})$, where the $L_i$ are general linear forms, $s \geq 3$, and $a_1,\ldots,a_s$ are positive integers.  Let $L$ be a general linear form,  and set $p = \max \{ j \ | \ \sum_{a_i \leq j} (j+1-a_i) \leq j \}$. 

\begin{itemize}

\item[(a)] For every $j$, the homomorphism $\times L^2 : [R/I]_{j-2} \rightarrow [R/I]_j$ has maximal rank.

\item[(b)] The homomorphism 
$\times L^3 : [R/I]_{j-3} \rightarrow [R/I]_j$ fails to have maximal rank in all degrees~$j$ if and only if 
\[
 \# \{ a_i \leq p +1 \} = p + 2 - \sum_{a_i \le p} (p+1 - a_i) \ge 4,  
\]
this number is even, and none of the $a_i$ equals $p+2$. Moreover, in this case multiplication fails maximal rank in exactly one degree, namely when $j = p+2$. Here we have $\dim [R/I]_{p-1} = \dim [R/I]_{p+2}$,  but the cokernel has dimension 1.

\end{itemize}
\end{theorem}

As a special case of part (b) of the above result, we give in Corollary \ref{cor:L3 and uniform power}  the result when all the $a_i$ are equal. This explains part of the computational evidence provided in Remark 6.2 of \cite{MM}. Theorem ~\ref{thm:intro1} also gives as corollaries statements on the SLP for ideals in three variables having the form $(L_0^{a_0},L_1^{a_1},\dots,L_s^{a_s})$ when $a_0=2$ (Theorem \ref{L2 in ideal}), and when $a_0=3$ and the $a_i$ are all equal for $i \geq 1$ (Corollary~\ref{cor:same degree}).

We also investigate the WLP of ideals in four variables.
Examples show that statements analogous to Theorem~\ref{thm:intro1}, but in rings with more variables, are not necessarily true (see Remark~\ref{no higher power works}). However, 
combining  the above result and the mentioned exchange property, we obtain results about the  WLP for ideals in four variables. In particular, we show (see Corollaries~\ref{four variables} and \ref{for:uniform powers in 4 var}). 

%
%
%
%
%
%

\begin{theorem}
Let $I = (L_0^{a_0}, L_1^{a_1},\dots, L_s^{a_s}) \subset R = K[x_1,x_2,x_3,x_4]$, with $s \geq 3$, where the $L_i$ are general linear forms. Then one has: 

\begin{itemize}

\item[(a)] If one of the powers $a_i$ is at most two, then $R/I$ has the WLP.

\item[(b)] If $a_0 = 3$ and $a_1 = \cdots = a_s = t \ge 3$, then $R/I$ has the WLP if and only if one of the following conditions holds:

\begin{itemize}
\item[(i)] $s$ is odd; 
\item[(ii)] $s$ is even and $t$ is not a multiple of $s-1$.
\end{itemize}

\end{itemize}
\end{theorem}

The paper is organized as follows. In Section 2 we recall some background results for ideals in three variables. The exchange property is discussed in Section 3. Previous approaches to the problems discussed here rely on involved computations. We avoid most of these by using the techniques presented in Section 4. Multiplication by $L^2$ and $L^3$ is discussed in Sections 5 and 6, respectively. Applications to the SLP and WLP are described in Section 7.

%
%


\section{Background}

We assume throughout that $K$ is a field of characteristic zero. We will consider ideals $I = (L_1^{a_1},\dots,L_s^{a_s}) \subset R = K[x,y,z]$ where the $L_i$ are general linear forms, and for convenience we will assume $a_1 \leq \dots \leq a_s$.

The following result was a launching point for \cite{MM} and for the current paper, which study what can be said when we consider multiplication by higher powers of a general linear form. Notice that unlike all results obtained in this paper, in the following the linear forms are not assumed to be general.

\begin{theorem}[\cite{SS}] \label{SS theorem}
If $I = (L_1^{a_1},\dots,L_s^{a_s}) \subset R = K[x,y,z]$ where the $L_i$ are arbitrary linear forms then $R/I$ has the WLP.
\end{theorem}

\begin{remark}
In \cite{DIV} and \cite{CHMN} it was shown that the result of Schenck and Seceleanu does not extend even to multiplication by the square of a general linear form. For this reason from now on we assume that the linear forms are general.
\end{remark}

\begin{theorem}[\cite{EI}] \label{emsalem-iarrobino}
Let $\langle L_1^{a_1} ,\dots,L_s^{a_s} \rangle \subset R$ be an ideal generated by powers of  $s$ linear forms with positive integers $a_1,\dots,a_n$.  Let $\wp_1, \dots, \wp_n$ be the ideals of the $s$ points in $\mathbb P^{2}$ that are dual to the linear forms.  Then one has,  for each integer $j$,
\[
\dim \left [R/ \langle L_1^{a_1}, \dots, L_n^{a_n} \rangle  \right ]_j =
\dim \left [ \bigcap_{a_i \le j}  \wp_i^{j-a_i +1} \right ]_j .
\]
\end{theorem}

Following \cite{MMN2} and \cite{MM}, from now on, we will denote by
$${\mathcal L}_{2}(j; b_1, b_2,\cdots ,b_n)$$
the linear system  $[ \wp_1^{b_1} \cap \dots \cap \wp_n^{b_n} ]_j\subset [R]_j $. Note that we view it as a vector space, not a projective space, when we compute dimensions.
If necessary, in order to simplify notation, we use superscripts to indicate repeated entries. For example,
$\mathcal L_2(j; 5^2, 2^3) = \mathcal L_2 (j; 5, 5, 2, 2, 2)$.

Notice that, for every linear system $\mathcal L_2 (j;b_1,\ldots,b_n)$, one has
\[
\dim_K  \mathcal L_2 (j; b_1,\ldots,b_n) \ge \max \left \{0, \binom{j+2}{2} - \sum_{i=1}^n \binom{b_i +1}{2} \right \},
\]
where the right-hand side is called the {\em expected dimension} of the linear system. If the inequality is strict, then the linear system $\mathcal L_2 (j; b_1,\ldots,b_n)$ is called {\em special}. It is a difficult problem to classify the special linear systems.

Using Cremona transformations, one can relate two different linear systems (see \cite{Nagata}, \cite{LU}, or \cite{Dumnicky}, Theorem 3), which we state only in the form we will need even though the cited results are more general.

\begin{lemma}
  \label{lem:Cremona}
Let $n >   2$ and let $j, b_1,\ldots,b_n$ be non-negative integers, with $b_1 \geq \dots \geq b_n$.  Set $m =  j - (b_1 + b_2 + b_{3})$. If $b_i + m \ge 0$ for all $i = 1,2,3$, then
\[
\dim_K \mathcal L_2 (j; b_1,\ldots,b_n) = \dim_K \mathcal L_2 (j + m; b_1 +m,b_2+m,b_{3} +m , b_{4},\ldots,b_n).
\]
\end{lemma}

The analogous linear systems have also been studied for points in $\mathbb P^r$. Following \cite{DL}, the linear system $\mathcal L_r (j; b_1,\ldots,b_n)$ is said to be in {\em standard form} if 
\[
(r-1)  j \ge b_1 + \dots  + b_{r+1} \quad \text{and} \quad b_1 \ge \cdots \ge b_n \ge 0.
\]
In particular, for $r=2$,  they show that every linear system in standard form is non-special.  (This is no longer true if $r \ge 3$. For example,  $\mathcal L_3 (6; 3^9)$ is in standard form and special.)

Notice again that we always use the vector space dimension of the linear system rather than the dimension of its projectivization. Furthermore, we always use the convention that a binomial coefficient $\binom{a}{r}$ is zero if $a < r$.

\begin{remark} \label{bezout}
B\'ezout's theorem also provides a useful simplification. Again, we only state the result we need in this paper. Assume the points $P_1,\dots,P_n$ are general.  
If $j < b_1 + b_2$ then
\[
\dim \mathcal L_2(j;b_1,\dots,b_n) = \dim \mathcal L_2 (j-1; b_1-1,b_2-1,b_3,\dots, b_n).
\]
\end{remark}

\begin{remark} \label{preserve}
In section \ref{L3 section} we will use Cremona transformations to complete some calculations. We observe here that if the linear system $\mathcal L_2(d; a_1,\dots,a_s)$ is transformed to the linear system $\mathcal L(e; b_1,\dots,b_s)$ by a Cremona transformation, then the fat point scheme $a_1 P_1 + \dots + a_s P_s$ imposes independent conditions on curves of degree $d$ if and only if $b_1 P_1 + \dots + b_s P_s$ imposes independent conditions on curves of degree $e$. Indeed, this can be shown by a tedious computation involving binomial coefficients, and Brian Harbourne has shown us a more elegant argument using the blow-up surface associated to the points $P_1,\dots,P_s$.
\end{remark}

We need the following special case of a theorem by Alexander and Hirschowitz on double points in any projective space \cite{AH}. 

\begin{theorem} 
     \label{thm:AH}
For a set of general double points in $\PP^2$, a linear system $\mathcal L(d; 2^m)$ has the 
expected dimension, namely $\max \{0, \binom{d+2}{2} - 3m\}$, unless $(d, m) \in \{ (4, 5), (2, 2)\}$. In these two  
cases the expected dimension is zero, but the actual dimension is one. 
\end{theorem}

\begin{remark} \label{bar r} 
If $0 \neq L \in R$ is a linear form and $k >0$, then we have for $\bar R = R/L^k R$ that 
$\dim [\bar R]_j = kj + 1 - \binom{k-1}{2}$ if $j \geq k$. If $j \leq k-1$, $\displaystyle \dim [\bar R]_j = \binom{j+2}{2}$. In particular, for $k=3$ we have $\dim [ \bar R]_j = 3j$ for $j \geq 1$ but not for $j=0$. On the other hand, for $k=2$ we have $\dim [\bar R]_j = 2j+1$ for all $j \geq 0$.
\end{remark}

\begin{remark} \label{binom}
For any integer $k \geq 0$,  we have $\displaystyle \binom{k+2}{2} - \binom{k}{2} = 2k+1$ and 
 $\displaystyle \binom{k+3}{2} - \binom{k}{2} = 3k+3$.
\end{remark}


\section{An exchange property}

We begin with a useful result relating the SLP for certain algebras to WLP for others. This idea has been studied in difference settings, e.g. \cite[Proposition 2.1]{MMN2}, but we formalize it here.

\begin{proposition}
    \label{lem:exchange}
Let $A$ be a standard graded artinian $K$-algebra over an infinite field $K$. Consider general elements $\ell, L \in [A]_1$. Fix positve integers $b$ and  $k$.  Then one has: 

\begin{itemize}

   \item[(a)] Assume $A$ has the WLP and  $b \ge k$.  If multiplication by $\ell^k$ on both $A$ and $A/L^b A$ has maximal rank (in each degree), then multiplication by $L^b$ has 
                  maximal rank on $A/\ell^k A$ (in each degree). 
   
   \item[(b)] If multiplication by $\ell^k$ on $A/L^b A$ has maximal rank (in each degree) and multiplication by $L^b$ on $A$  has maximal rank (in each degree), then multiplication by 
                  $L^b$ has maximal rank on $A/\ell^k A$ (in each degree). 

\end{itemize} 
\end{proposition}

\begin{proof}  
(a) For any integer $j$ consider the commutative diagram
\begin{align} 
   \label{eq:diagram}
\begin{CD}
[A]_{j-k} @>\ell^k>> [A]_{j} @>>> [A/\ell^k]_{j} @>>> 0 \\
@ VV{L^b}V   @ VV{L^b = \psi}V  @ VV{L^b = \ffi}V  \\
 [A]_{j-k+b} @>\ell^k>> [A]_{j+b} @>>> [A/\ell^k]_{j+b} @>>> 0 \\
 @ VVV @  VVV @  VVV \\
 [A/L^b A]_{j-k+b} @>\ell^k = \gamma>> [A/L^b A]_{j+b} @>>> [A/(\ell^k, L^b)A]_{j+b} @>>> 0 \\
  @ VVV @  VVV @  VVV \\
  0 && 0 && 0. 
  \end{CD}
\end{align}
We want to show that the rightmost vertical multiplication map, call it $\ffi$,  has maximal rank. This is clear if  any of the modules in the vertical sequence on the right is zero. Assume none of these modules is trivial. Since the three horizontal multiplication maps have maximal rank by assumption it follows that they are injective. Thus, the Snake Lemma shows that injectivity of the map $\ffi$ follows, once we know that the vertical multiplication in the middle, call it $\psi$,  is injective. 

Since $A$ has the WLP its Hilbert function is unimodal. Let $q$ be the largest integer such that $\dim [A]_{q-1} \le \dim [A]_q$.  The fact that the map 
$\times \ell^k : [A]_{j-k+b} \rightarrow [A]_{j+b}$ is injective forces $j-k+d < q$. Thus, the Hilbert function of $A$ is non-decreasing on the closed interval from zero to $j-k+d$. 
The WLP  of $A$ now gives that $\times L^{b-k}: [A]_j \to [A]_{j-k+b}$ is a composition of $b-k$ injective multiplications by $L$. Injectivity of the middle multiplication by $\ell^k$ 
implies that  $\times L^k : [A]_{j-k+b} \rightarrow [A]_{j+b}$ is injective as well. Hence, as composition of injective maps,  $\psi$ is injective, as desired. 

(b) This is easier and similar to the proof of \cite[Proposition 2.1]{MMN2}. Again we may assume that each of the  modules in the vertical sequence on the right is not zero. Since the bottom multiplication by $\ell^k$, call it $\gamma$, has maximal rank by assumption it must be injective. Furthermore, we assumed that $\psi$ has maximal rank. If it is surjective, then so is $\ffi$. If $\psi$ is injective, then the Snake Lemma yields that $\ffi$ also is injective. 
\end{proof}

\begin{corollary}
    \label{cor:square}
If both $A$ and $A/\ell^2$ have the WLP and multiplication by $\ell^2$ has maximal rank on both $A$ and $A/L^b A$ for every $b \ge 2$, then $A/\ell^2 A$ has the SLP. 
\end{corollary} 


\section{Some calculations} \label{gen results}

Before presenting  the main results in the coming sections, we give some preparatory calculations that will simplify the later arguments.

\begin{remark} \label{narrow down degree}
Let $R = K[x,y,z]$ and $I = (L_1^{a_1},\dots, L_s^{a_s})$ where $s \geq 3$. Let $\underline{h} = (h_i)_{0 \leq i \leq d}$ be the Hilbert function of $R/I$, where $d$ is the last degree where $R/I$ is not zero. Since $R/I$ has the weak Lefschetz property by  Theorem \ref{SS theorem}, it follows (see \cite{HMNW}) that  there is a  unique integer $p$ such that
\[
h_{p-1} < h_p \geq h_{p+1}.
\]
Moreover, we get,  for a general linear form $L$, that the multiplication map $\times L : [R/I]_{j-1} \rightarrow [R/I]_j$ is injective for $j \leq p$ and surjective for $j \geq p+1$.  We will use this to narrow down our study of the multiplication maps $\times L^2$ and $\times L^3$. 

First we consider $\times L^2$. The above observations imply  that $\times L^2 : [R/I]_{j-2} \rightarrow [R/I]_j$ is injective if $j \leq p$ and surjective if $ j \ge p+2$. Thus, for studying  whether $\times L^2 : [R/I]_{j-2} \rightarrow [R/I]_j$ has maximal rank for all $j$, it remains to consider the map $\times L^2 : [R/I]_{p-1} \rightarrow [R/I]_{p+1}$. 

Similarly, if we are studying whether $\times L^3 : [R/I]_{j-3} \rightarrow [R/I]_j$ has maximal rank for all $j$,  then we know that this is true if $j \le p$ or $j \ge p+3$. Thus, it remains to 
consider the map if  $j=p+1$ or $j=p+2$. 
\end{remark}

The following result will simplify many of the calculations in the rest of the paper. 

\begin{proposition} 
   \label{magic}
Let $R = K[x,y,z]$ and let $I = (L_1^{a_1}, \dots, L_s^{a_s}) \subset R$, where the $L_i$ are linear forms and $2 \leq a_1 \leq \dots \leq a_s$.  Let $L$ be a general linear form and 
let $k \geq 1$ and $j \geq \max \{ k, a_s \}$ be  positive integers. If
{\small 
\[
\dim [R/(I,L^k)]_j = \left [ jk + 1 - \binom{k-1}{2} \right ] - \sum_{a_i \le j-k} \left [ k(j-a_i) + 1 - \binom{k-1}{2}  \right ]  - \sum_{a_i > j-k} \binom{j-a_i + 2}{2}
\]  }
then $\times L^k : [R/I]_{j-k} \rightarrow [R/I]_j$ is injective. Furthermore, $I$ is minimally generated by the given $s$ powers. 
\end{proposition}

\begin{proof}
Consider the  syzygy module $E$ of $I$, defined by the  exact sequence
\[
0 \rightarrow E \rightarrow \bigoplus_{i=1}^s R(-a_i) \stackrel{\phi}{\longrightarrow} R \rightarrow R/I \rightarrow 0
\]
where $\phi $ is the matrix $ [L_1^{a_1} \ \dots \  L_s^{a_s}]$. 
This fits in the following commutative diagram, where $\bar R = R/(L^k)$.
\begin{equation} \label{commutative diagram}
\begin{array}{cccccccccccc}
&&  && 0 && 0 \\
&&  && \downarrow && \downarrow \\
0 & \rightarrow & E(-k) & \rightarrow & \bigoplus_{i=1}^s R(-a_i -k) & \rightarrow & R(-k) & \rightarrow & R/I (-k) & \rightarrow & 0 \\
&&  && \phantom{ {{\hbox{$\scriptstyle \times L^k$}}} } \downarrow  {\hbox{$\scriptstyle \times L^k$}} &&  \phantom{ {{\hbox{$\scriptstyle \times L^k$}}} } \downarrow  {\hbox{$\scriptstyle \times L^k$}} \\
0 & \rightarrow & E & \rightarrow & \bigoplus_{i=1}^s R(-a_i ) & \rightarrow & R & \rightarrow & R/I  & \rightarrow & 0 \\
&&  && \downarrow && \downarrow \\
0 & \rightarrow & A & \rightarrow & \bigoplus_{i=1}^s \bar R(-a_i ) & \rightarrow & \bar R & \rightarrow & B  & \rightarrow & 0 \\
&&&& \downarrow && \downarrow \\
&&&& 0 && 0
\end{array}
\end{equation}
where for now $A$ and $B$ are just the kernel and cokernel of the  map given by the restriction of $\phi$ to $\bar R$.  The Snake Lemma gives the long exact sequence
\begin{equation} \label{snake lemma}
0 \rightarrow E(-k) \stackrel{\times L^k}{\longrightarrow} E \rightarrow A \rightarrow R/I(-k) \stackrel{\times L^k}{\longrightarrow} R/I \rightarrow B \rightarrow 0.
\end{equation}
(The fact that $B$ is both the cokernel of $\times L^k$ and of the bottom horizontal sequence in the commutative diagram was noted in \cite{DI}, Theorem 2.11.) Now using 
Remark \ref{bar r}, the bottom sequence in the commutative diagram gives that $[A]_j = 0$, so the result follows from (\ref{snake lemma}).
\end{proof}

For the proof of Theorem \ref{xL2} we do not actually need to know the precise value of $p$ as in Remark \ref{narrow down degree}. However, the following characterization of $p$ will be useful. 

\begin{proposition} \label{formula for p}
Let $I = (L_1^{a_1},\dots, L_s^{a_s})$. 
Let $A = R/I$ and let $L$ be a general linear form. Then 
\[
p = \reg (A/LA) = \max \left  \{ j \ | \ \sum_{a_i \leq j} (j+1-a_i) \leq j  \right \}
\]
\end{proposition}

In the above statement, it is understood that for $j  < \min \{ a_i\}$ the sum is zero so such a $j$ satisfies the inequality. It may even be the maximum. For example, if $s=6$ and $a_i = 5$ for all $i$, we get $p=4<a_i$.

\begin{proof}
The first equality follows because $A$ has the WLP \cite{SS}, so we only prove the second one. We note that $K[x,y,z]/L \cong K[x,y]$. Thus we have
\begin{equation*}
\begin{split}
\dim [A/LA]_j & =  \dim \mathcal L_1 (j; j+1-a_1, \dots, j+1-a_s) \\ \\
& =  \displaystyle \max \left \{ 0, j+1-  \sum_{j \geq a_i} (j+1-a_i)  \right \}
\end{split}
\end{equation*}
from which the result follows.
\end{proof}

\begin{corollary} \label{formula for p - uniform}
Assume that $a_i = t$ for all $i$. Then
\[
p =  \begin{cases} 
\left \lfloor \frac{s(t-1)}{s-1} \right \rfloor & \text{ if } s \le t \\
t - 1 & \text{ if } s \ge t+1. 
\end{cases}
\]

\end{corollary}

\begin{proof}
It follows from Proposition \ref{formula for p} after a standard computation.
\end{proof}

We now write the formula for $p$ in a more convenient way.

\begin{notation} \label{n_j notation}
For an $s$-tuple $N = (a_1,\ldots,a_s) \in \NN_0^s$   of non-negative integers, define 
\[
n_j := n_j (N) =  \# \{a_i \in N \; | \;  a_i \le j \} 
\]
and 
\[
p (N) := \max \left  \{ j \ | \ \sum_{a_i \leq j} (j+1-a_i) \leq j  \right \}
\]
as above. If $A = R/(L_1^{a_1},\dots,L_s^{a_s})$ and $N = \{ a_1,\dots,a_s \}$ we also write $p(A)$ for $p(N)$.
\end{notation}

\begin{lemma}
     \label{lem:rewrite}
For every integer $j$, one has 
\[
 \sum_{a_i \le j}  (j+1-a_i) = n_j + n_{j-1} + \cdots + n_0. 
\]    
\end{lemma}

\begin{proof}
Defining $n_{-1} := 0$, we compute 
\begin{equation*}
\begin{split}
\sum_{a_i \le j} (j+1-a_i) & = \sum_{k=0}^j \sum_{a_i = k} (j+1-k) \\
& = \sum_{k=1}^j (j+1-k) (n_k - n_{k-1}) \\
& = n_j + n_{j-1} + \cdots + n_0.
\end{split}
\end{equation*} 
\end{proof}

It follows that $p = p (N)$ is  defined by 
\begin{equation}
   \label{eq:def p}
n_0 +  \cdots + n_p \le p \quad \text{and} \quad n_0 +  \cdots + n_p + n_{p+1}  \ge p+2.
\end{equation}
In particular, $n_{p+1}$ must be positive. 

For simplicity of notation, we write $N + b$ for the $(s+1)$-tuple $(a_1,\ldots,a_s, b) $, where $b \ge 0$ is an integer. 
We now record the following observation. 

\begin{lemma}
     \label{lem:p not changing}
For every integer $b \ge 0$, one has: 
\begin{itemize}

    \item[(a)] $p(N+b) \le p (N)$. 
    
    \item[(b)]  Suppose $p = p (N) \ge 2$. If $n_{p-1} = 0$  and $n_p \le 1$, then $p (N+b) = p$ for every integer $b \ge 2$.  

\end{itemize}
\end{lemma}

\begin{proof} 
This is clear if $b \ge p+2$. Set $n_j' := n_j (N + b)$. Then the definition of $p$ gives 
$
n_2' +  \cdots + n_p' + n_{p+1}'  \ge p+2$. This implies claim (a). 

Moreover, if the assumptions of (b) are satisfied, we get for every $b \ge 2$ that $n_2' +  \cdots + n_p'  \le 1 + \cdots 1 + 1 + n_p \le p$, which implies $p (N+b) = p$.
\end{proof}


\section{Multiplication by $L^2$} \label{L2 section}

We will generalize the following theorems (changing the notation for consistency with our work here).

\begin{proposition} [\cite{MMN2}, Proposition 4.7]
Let $L_1,\dots,L_4,L$ be general linear forms in $K[x,y,z]$. Let $I = (L_1^{a_1}, \dots, L_4^{a_4})$. Then for each $j$, the homomorphism $\times L^2 : [R/I]_{j-2} \rightarrow [R/I]_j$ has maximal rank.
\end{proposition}

\begin{theorem}[\cite{MM},  Theorem 4.4]
Let $L_1, \dots,L_s, L$ be general linear forms in $K[x,y,z]$, where $s \geq 4$. Let $I = (L_1^t,\dots,L_s^t)$ for $t \geq 2$. Then for each $j$, the homomorphism $\times L^2 : [R/I]_{j-2} \rightarrow [R/I]_j$
has maximal rank.
\end{theorem}

In this section we will give a  simpler proof of a more general result, using the methods in the last section. It confirms a conjecture made in \cite{MM}. After the first version of this paper was completed, R.M.\ Mir\'o-Roig kindly informed us that A. Andrade  and C. Almeida  \cite{AA} also proved the following result.


\begin{theorem} \label{xL2}
Let $R = K[x,y,z]$ and let $I = (L_1^{a_1}, \dots, L_s^{a_s})$, where the $L_i$ are general linear forms,  and  $s \geq 3$.   Let $L$ be a general linear form. 
Then for any $j$, the homomorphism $\times L^2 : [R/I]_{j-2} \rightarrow [R/I]_j$ has maximal rank.
\end{theorem}

\begin{proof}
Without loss of generality we can (and will) assume that $a_i \leq p+1$ for all $i$ (by Remark \ref{narrow down degree}, by removing generators of larger degree). 
Our goal is to show that  for each $j$, either the kernel or the cokernel of the homomorphism is zero, and we know by Remark \ref{narrow down degree} that it is enough to show it for $j=p+1$. What we will show is a bit stronger, namely that either $[A]_{p+1} = 0$ or $[B]_{p+1} = 0$, where $A$ and $B$ are given in the commutative diagram (\ref{commutative diagram}) and $p$ is the integer defined in  Remark \ref{narrow down degree}. This will prove what we need.

We now make a formal computation. At several steps the validity of the assertion needs to be explained, and will be addressed afterwards; essentially whenever the step is not valid, it will turn out that this means that we are done with what we want to show. We set $D = \sum_{i=1}^s a_i $. 
By Theorem~\ref{emsalem-iarrobino} we have
\begin{equation*}
\begin{split}
\dim [B]_{p+1}
& =  \dim [R/(L_1^{a_1},\dots, L_s^{a_s},L^2)]_{p+1} \\ 
& =  \dim \mathcal L (p+1; p-a_1+2,\dots, p-a_s+2, p) \\ 
& =   \dim \mathcal L_2 (p+1 -\sum_{i=1}^s (p-a_i+1); 1^s, p - \sum_{i=1}^s (p-a_i+1) ) \\ 
& =    \dim \mathcal L_2 (p+1 -\sum_{i=1}^s (p-a_i+1); p - \sum_{i=1}^s (p-a_i+1) ) - s \\ 
& =   \binom{p+3-\sum_{i=1}^s (p-a_i +1)}{2} - \binom{p+1-\sum_{i=1}^s (p-a_i +1)}{2} -s \\ 
& =   2 \left [ p+1 - \sum_{i=1}^{s} (p-a_i+1) \right ] +1-s \\ 
& =  2 \left [p+1-s(p+1) +D \right ] +1-s \\ 
& =  2(p+1) +1-s - 2s(p+1)+2D.
\end{split}
\end{equation*}

We now comment on key steps in the above formal computation.

\begin{itemize}
\item[Line 2:] Notice that by our assumption that $a_i \leq p+1$ (Remark \ref{narrow down degree}), we have $p-a_i+2 \geq 1$ for all $i$. 

\smallskip

\item[Line 3:] The goal is to reduce all of the multiplicities to 1 except for the last one, using Bezout's theorem (Remark \ref{bezout}), splitting off lines joining the point with (original) multiplicity $p$ to each of the other points, one at a time. Since the last multiplicity is always one less than the degree of the curves, this can be done as long as the degree of the curves is positive and the other multiplicity is $> 1$. Thus if the degree remains positive throughout, line 3 is valid. If it ever happens that the degree reduces to zero, this dimension is zero. But this means that $\times L^2$ is surjective in the desired degree, which is what we wanted to show. So we can assume that line 3 is valid.

\smallskip

\item[Line 4:] Since all the linear forms are general, the dual points are general. But $t$ general simple points impose up to $t$ independent conditions on any non-empty linear system. If we reach a point where the linear system is empty, again the dimension is zero and we are done. So again, we can assume that line 4 is valid.

\smallskip

\item[Line 5:] Given a complete linear system of plane curves of degree $d$, the requirement that the curves have multiplicity $m$ at any point $P$ imposes $\binom{m+1}{2}$ independent conditions. 

\smallskip

\item[Line 6:] This uses Remark \ref{binom}.
\end{itemize}

Now we compute the right-hand side of the equation in Proposition \ref{magic}. 
\[
2(p+1)+1 - \sum_{i=1}^s [2(p+1-a_i)+1]  = (2p+3) -  2s(p+1) + 2D - s .
\]
Comparing with the computation that we have previously completed, we  are done by using Proposition \ref{magic}. Thus $\times L^2$ has maximal rank.
\end{proof}


\section{Multiplication by $L^3$} \label{L3 section}

We recall the following theorem.

\begin{theorem}[\cite{MM} Theorem 5.1] \label{MM 5.1}
Let $L_1, \dots,L_4, L$ be general linear forms in $K[x,y,z]$. Let $I = (L_1^t,\dots,L_4^t)$ for $t \geq 3$. 

\begin{itemize}
\item[(i)] If $t \equiv 0 \ (\hbox{mod } 3)$, set $t = 3t_0$. Then for $\delta = 4t_0$ we have $\dim [R/I]_{\delta-3} = \dim [R/I]_\delta$, and $\times L^3$ fails by exactly one to be an isomorphism between these components. In all other degrees, $\times L^3$ has maximal rank.

\item[(ii)] If $t \not \equiv 0 \ (\hbox{mod } 3)$ then $\times L^3$ has maximal rank in all degrees.
\end{itemize}
\end{theorem}

\noindent We note that for the integer $p$ defined in Remark \ref{narrow down degree},  Corollary \ref{formula for p - uniform} gives for Theorem \ref{MM 5.1} that in case (i) we have $\delta = p+2$. Our two main results in this section extend this result in two ways: by increasing the number of linear forms and by allowing mixed powers. We begin with the latter. We continue to use the invariant $p$ defined in Remark \ref{narrow down degree} and discussed in Proposition \ref{formula for p}. We have seen that to show that $\times L^3 : [R/I]_{j-3} \rightarrow [R/I]_j$ has maximal rank for all $j$, it is enough to show it for $j=p+1$ and $j=p+2$. First we show that the former is always true.

\begin{theorem} \label{xL3} 
Let $R = K[x,y,z]$ and let $I = (L_1^{a_1}, \dots, L_s^{a_s})$, where the $L_i$ are general linear forms  and  $s \geq 3$. Set 
 $p = \max \left  \{ j \ | \ \sum_{a_i \leq j} (j+1-a_i) \leq j  \right \}$. Let $L$ be a general linear form. 
Then the homomorphism $\times L^3 : [R/I]_{p-2} \rightarrow [R/I]_{p+1}$ has maximal rank. 
\end{theorem}

\begin{proof}

Let $m = \# \{ a_i \leq p \}$ and $n = \# \{ a_i = p+1 \}$ and put $A = R/I$. We may assume $p \ge 3$. 

Again we make a formal calculation followed by comments of justification. By Theorem~\ref{emsalem-iarrobino} we have
\begin{equation*}
\begin{split}
\dim [A/L^3 A]_{p+1}
& =  \dim [R/(L_1^{a_1},\dots, L_s^{a_s},L^3)]_{p+1} 
\\ 
& =  \dim \mathcal L_2 (p+1;  p-1 ,  p+2-a_i \; | \; a_i \le p+1) 
\\ 
& = \dim \mathcal L_2 (p+1 ;  p-1, 1^n,   p+2-a_i \; | \; 1 \le  i \le m) 
\\
& =   \dim \mathcal L_2 (p+1 -\sum_{i=1}^m (p-a_i); p -1 - \sum_{i=1}^m (p-a_i), 2^m, 1^n ) 
\\ 
& = \dim \mathcal L_2 (d+2; d, 2^m) - n, 
\end{split}
\end{equation*} 
where $d = p -1 - \sum_{i=1}^m (p-a_i)$. 

We now comment on key steps in the above formal computation.

\begin{itemize}
\item[Line 3:] This uses the definition of the numbers $m$ and $n$. We  note that it is possible that $m=0$ and all relevant generators have degree $p+1$. In this case line 3 is 
\[
\dim \mathcal L_2(p+1; p-1,1^n) = \binom{p+3}{2} - \binom{p}{2} - n = 3p+3-n,
\]
and one easily checks that this is equal to the right-hand side of the condition in Proposition \ref{magic}. Hence $\times L^3 : [R/I]_{p-2} \rightarrow [R/I]_{p+1}$ is injective. We thus can assume from now on that $m \geq 1$.

\smallskip

\item[Line 4:] The goal is to reduce all of the multiplicities to either 1 or 2 except for the first one, using Bezout's theorem (Remark \ref{bezout}), splitting off lines joining the point with (original) multiplicity $p-1$ to each of the other points, one at a time. Since the first multiplicity is always two less than the degree of the curves, this can be done as long as  at least one of the $m$ multiplicities corresponding to some $a_i \leq p$ is greater than 2. Notice that $d \ge 0$. Indeed, the definition of $p$ gives $p \ge \sum_{a_i \leq p} (p+1-a_i)$. It follows that 
\begin{equation} 
   \label{eq:sum estimate}
\sum_{i=1}^m (p-a_i) \le p - m,  
\end{equation} 
which implies $d \geq m-1 \geq 0$ as claimed.

\smallskip

\item[Line 5:] Since all the linear forms are general, the dual points are general. But $n$ general simple points impose $n$ independent conditions on any non-empty linear system. If we reach a point where the linear system is empty, again the dimension is zero and we are done. So again, we can assume that line 5 is valid.
\end{itemize} 

We now rewrite the system $\mathcal L_2 (d+2; d, 2^m)$ using Cremona transformations in order to show that it has the expected dimension. 
\smallskip 

\emph{Case 1}: Suppose $d \ge m$. Write $m= 2 k + \delta$ with integers $k \geq 0$ and 
$\delta \in \{0, 1\}$. Using $k$ Cremona transformations we get 
\[
\dim \mathcal L_2 (d+2; d,  2^{m}) = \dim \mathcal L_2 (d+2 - 2k; d - 2k,  2^{\delta}). 
\]
The linear system on the right-hand side has the expected dimension, and thus so does $\mathcal L_2 (d+2; d,  2^{m})$ by Remark \ref{preserve}. 
\smallskip 

\emph{Case 2}: Since $d \ge m-1$ by Inequality \eqref{eq:sum estimate}, we are left with the case where $d = m-1$. Writing $d = 2 k + \delta$ with integers $k$ and 
$\delta \in \{0, 1\}$ and using $k$ Cremona transformations we get 
\[
\dim \mathcal L_2 (d+2; d,  2^{m}) = \dim \mathcal L_2 (\delta+2; \delta,  2^{1 +\delta}). 
\]
Hence both systems have  expected dimension. 
\smallskip 

The above discussion shows that 
\begin{equation*}
\begin{split}
\dim [A/L^3 A]_{p+1}
& =   3 d + 6 - 3 m  - n  = 3 \big [ p+1 - \sum_{i=1}^m (p-a_i) - m \big ] - n . 
\end{split}
\end{equation*} 
Comparing with Proposition \ref{magic}, we conclude that $\times L^3 : [A]_{p-2} \rightarrow [A]_{p+1}$ is injective. \end{proof}

As noted in Theorem \ref{MM 5.1}, it is not always true that $\times L^3$ has maximal rank, although Theorem \ref{xL3} shows that it can only fail for the map $\times L^3 : [R/I]_{p-1} \rightarrow [R/I]_{p+2}$. From now on we use Notation \ref{n_j notation} and Lemma \ref{lem:rewrite}.

\begin{theorem} \label{xL3 general}
Let $R = K[x,y,z]$ and let $I = (L_1^{a_1}, \dots, L_s^{a_s})$, where the $L_i$ are general linear forms,  and 
 $s$ and the $a_i$\!'s are non-negative integers. Set 
 \[
 p = \max \left  \{ j \ | \ \sum_{a_i \leq j} (j+1-a_i) \leq j  \right \} = \max \left \{ j \ | \  n_0 + n_1 + \dots + n_j  \leq j \right \}
 \] 
Let $L$ be a general linear form. Then the homomorphism 
$\times L^3 : [R/I]_{j-3} \rightarrow [R/I]_j$ fails to have maximal rank in all degrees $j$ if and only if 
\[
n_{p+1} =   p + 2 - \sum_{a_i \le p} (p+1 - a_i) \ge 4,  
\]
this number is even, and none of the $a_i$ equals $p+2$. Moreover, in this case multiplication fails maximal rank in exactly one degree, namely when $j = p+2$. Here we have $\dim [R/I]_{p-1} = \dim [R/I]_{p+2}$,  but the cokernel has dimension 1.

\end{theorem}

\begin{proof}
We may assume that each $a_i \ge 2$ because otherwise $A = R/I$ has the strong Lefschetz property by \cite{HMNW}. If $s \leq 2$ then the maps are injective, and if $s=3$ then after a change of variables we can assume that $I$ is a monomial  complete intersection, so the result of \cite{stanley} and of \cite{watanabe} applies, and $R/I$ has the strong Lefschetz property. Thus we can assume that $s \ge 4$. By Proposition \ref{formula for p}, we know that $p = \reg (A/LA)$. Hence, Remark \ref{narrow down degree} shows that the multiplication map $\times L^3 : [A]_{j-3} \rightarrow [A]_j$ has maximal rank if $j \notin \{p+1, p+2\}$. By Theorem \ref{xL3}, the maximal rank property is also true if $j = p+1$. Hence it remains to consider $\times L^3 : [A]_{p-1} \rightarrow [A]_{p+2}$, where $p \ge 2$.

Let $m = n_p = \# \{ a_i \leq p \}$,  $n = \# \{ a_i = p+1 \}$, and $q =  \# \{ a_i = p+2 \}$.
Again we make a formal calculation followed by comments of justification. By Theorem~\ref{emsalem-iarrobino} we have
\begin{equation*}
\begin{split}
\dim [A/L^3 A]_{p+2}
& =  \dim [R/(L_1^{a_1},\dots, L_s^{a_s},L^3)]_{p+2} 
\\ 
& =  \dim \mathcal L_2 (p+2; p,  p+3-a_i \; | \; a_i \le p+2) 
\\ 
& =  \dim \mathcal L_2 (p+2; p, 2^n, 1^q,   p+3-a_i \; | \; 1 \le i \le m) 
\\ 
& =   \dim \mathcal L_2 (p+2 -\sum_{i=1}^m (p+1-a_i); p - \sum_{i=1}^m (p+1-a_i),  2^{m+n}, 1^q )
 \\ 
& =  \max \{0,\   \dim \mathcal L_2 (d+2; d,  2^{m+n}) - q\},  
\end{split}
\end{equation*}
where $d = p - \sum_{i=1}^m (p+1-a_i)$. 

We now comment on the steps in this formal computation.

\begin{itemize}
\item[Line 3:] This uses the definition of the numbers $m, n$, and $q$.  We note that it is possible that  $m=0$, i.e. that the smallest generator has degree $p+1$. In this case line 3 indicates that we have to compute $\dim \mathcal L_2(p+2; p, 2^n, 1^q)$. This does not necessarily have the expected dimension; e.g. if $p=2$, $n=4$ and $q=0$, we know that this dimension is 1 rather than 0, and this is an exception allowed in the statement of the theorem. But the argument below goes through also in this situation.

\smallskip

\item[Line 4:] The goal is to reduce all of the multiplicities, other than the first one, that are at least 3 to 2.  This can be done by  using Bezout's theorem (Remark \ref{bezout}), splitting off lines joining the point with (original) multiplicity $p$ to each of the other points, one at a time. Since the first multiplicity is always two less than the degree of the curves, this can be done as long as  at least one of the $m$ multiplicities corresponding to some $a_i +2 - p$ is at least 3.  (Again, if $m=0$ there is nothing to do.) 
Notice that $d = p - \sum_{i=1}^m (p+1-a_i)$ is not negative by the definition of $p$. 

\smallskip

\item[Line 5:] As before, this uses that the $q$ points with multiplicity one are general. 
\end{itemize}

Our next goal is to determine when the linear system $\mathcal L_2 (d+2; d,  2^{m+n})$ has the 
expected dimension by using Cremona transformations. This is clearly true if $d \leq 1$, and not hard to check if $m+n \leq 1$, so we can assume without loss of generality that $d \ge 2$ and $m+n \ge 2$. Then such a 
transformation takes $\mathcal L_2 (d+2; d,  2^{m+n})$ to $\mathcal L_2 (d; d-2,  2^{m+n-2})$. We 
consider two cases. 
\smallskip 

\emph{Case 1}: Suppose $d > m+n$. Write $m + n = 2 k + \delta$ with integers $k$ and 
$\delta \in \{0, 1\}$. Using $k$ Cremona transformations we get 
\[
\dim \mathcal L_2 (d+2; d,  2^{m+n}) = \dim \mathcal L_2 (d+2 - 2k; d - 2k,  2^{\delta}). 
\]
The linear system on the right-hand side has the expected dimension, and thus so does $\mathcal L_2 (d+2; d,  2^{m+n})$ by Remark \ref{preserve}. 
\smallskip 

\emph{Case 2}: Suppose $d \le m+n$. This time write  $d = 2 k + \delta$ with integers $k$ and 
$\delta \in \{0, 1\}$. Again using  $k$ Cremona transformations we get 
\[
\dim \mathcal L_2 (d+2; d,  2^{m+n}) = \dim \mathcal  L_2 (\delta+2; \delta,  2^{m+n-d +\delta}). 
\]
Theorem \ref{thm:AH} shows that the linear system on the right-hand side has the expected dimension, unless $\delta = 0$ and $m+n-d = 2$, which is equivalent to the condition 
\begin{equation}
    \label{eq:exceptionalCase} 
    m+n = p+2 - \sum_{i=1}^m (p+1-a_i) \quad \text{is even}. 
\end{equation} 
In the exceptional case we have $\dim \mathcal L_2 (d+2; d,  2^{m+n}) = 1$. This concludes Case 2. 
\smallskip

The above discussion of cases shows that if Condition \eqref{eq:exceptionalCase} is not satisfied, then $\mathcal L_2 (d+2; d,  2^{m+n})$ has the expected dimension. This gives 
\begin{equation*}
\begin{split}
\dim [A/L^3 A]_{p+2}
& =   \max \{0,\ 3 d + 6 - 3 (m+n) - q\} 
\\
& =  \max \{0,\ 3 \big [ p+2 - \sum_{i=1}^m (p+1-a_i) - m - n \big ] - q\}. 
\end{split}
\end{equation*} 
If the right-hand side is zero, then $\times L^3 : [A]_{p-1} \rightarrow [A]_{p+2}$ is surjective, and we are done. Otherwise, we get
\begin{equation*}
\begin{split}
\dim [A/L^3 A]_{p+2}
& =   3 \big [ p+2 - \sum_{i=1}^m (p+1-a_i) - m - n \big ] - q\} 
\\
& =  3 \left [ p+2 - \sum_{a_i \le p} (p+2-a_i) - \sum_{a_i = p+1} (p+2-a_i) \right ] - \sum_{a_i = p+2} \binom{p+2-a_i+2}{2}. 
\end{split}
\end{equation*} 
Comparing with Proposition \ref{magic}, we conclude that $\times L^3 : [A]_{p-1} \rightarrow [A]_{p+2}$ is injective. 

It remains to consider the case where Condition \eqref{eq:exceptionalCase} is satisfied. Since in this case $\dim \mathcal L_2(d+2; d,2^{m+n}) = 1$, we have 
\[
\dim [A/L^3 A]_{p+2} =  \max \{0,\ 1 - q\}.  
\]
Thus, if $q \ge 1$ the map $\times L^3 : [A]_{p-1} \rightarrow [A]_{p+2}$ is surjective, as desired. 

Finally, suppose $q = 0$ in addition to Condition \eqref{eq:exceptionalCase}. Then $\dim [A/L^3 A]_{p+2} = 1$. If $m + n = 2$, then $I$ has exactly two generators with a degree that is at most $p+2$. By generality they form a regular sequence, so  $\times L^3 : [A]_{p-1} \rightarrow [A]_{p+2}$ has maximal rank. Thus, for the remainder of the proof we may also assume that $m + n \ge 4$. Notice that $ \# \{ a_i \leq p +1 \} = m + n$. Therefore the argument is complete, once we have shown that $[A]_{p+2}$ and $[A]_{p-1}$ have the same dimension. 

In order to check this,  consider first  $[A]_{p-1}$.  Theorem~\ref{emsalem-iarrobino} gives 
\[
\dim [A]_{p-1} = \dim  \mathcal L_2 (p-1;   p-a_i \; | \; 1 \le i \le m). 
\]
We have seen above that $p-1 \ge \sum_{i=1}^m (p-a_i)$. Thus, the linear system on the right-hand side is in standard form and has the expected dimension. It follows that 
\[
\dim [A]_{p-1} = \binom{p+1}{2} - \sum_{i=1}^m \binom{p+1 -a_i}{2}. 
\]

We now consider $[A]_{p+2}$. Using $q = 0$,  Theorem~\ref{emsalem-iarrobino} gives 
\[
\dim [A]_{p+2} = \dim  \mathcal L_2 (p+2; 2^n,   p+3-a_i \; | \; 1 \le i \le m). 
\]
We claim that the system on the right-hand side has the expected dimension. Since $p+2 =  \sum_{i=1}^m (p+2-a_i)   + n$ by  Condition \eqref{eq:exceptionalCase}, this is certainly true if $m \ge 3$ and $n \ge 3$ because then 
\[
p+2 = \sum_{i=1}^m (p+2-a_i) + n \geq \sum_{i=1}^3 (p+2-a_i) + 3 = \sum_{i=1}^3 (p+3-a_i)
\]
and the system is in standard form. 

If $m = 3$ then it remains to consider the case where  $n =1$. Now Condition \eqref{eq:exceptionalCase} shows that we can apply a Cremona transformation to obtain 
\[
\dim  \mathcal  L_2 (p+2; 2,   p+3-a_i \; | \; 1 \le i \le 3) = \dim  \mathcal  L_2 (p; 2,   p+1-a_i \; | \; 1 \le i \le 3). 
\]
The new system is in standard form by the definition of $p$, and thus also $\mathcal L_2 (p+2; 2,   p+3-a_i \; | \; 1 \le i \le 3)$ has the expected dimension. 

Suppose $m =2$. Then Condition \eqref{eq:exceptionalCase} gives $p+2 -   \sum_{i=1}^2 (p+3-a_i) - 2 = n - 4$. Hence, $ \mathcal  L_2 (p+2; 2^n,   p+3-a_i \; | \; 1 \le i \le 2)$ is in standard form if $ n\ge 4$. Otherwise, $n = 2$ and a Cremona transformation gives
\[
\dim  \mathcal  L_2 (p+2; 2^2,   p+3-a_i \; | \; 1 \le i \le 2) = \dim  \mathcal  L_2 (p; 2,   p+1-a_i \; | \; 1 \le i \le 2). 
\]
Again the new system is in standard form, and thus the original system has the expected dimension by Remark \ref{preserve}. 

If $m = 1$, then $n \ge 3$ and the linear system $\mathcal  L_2 (p+2; 2^n,   p+3-a_1)$ is in standard form if  $n \ge 5$. If $n=3$ a Cremona transformation shows that the system has the expected dimension as well. 

If $m = 0$, then Condition \eqref{eq:exceptionalCase} gives $p+2 = n$. Since we must have $n \ge 4$, the system $\mathcal  L_2 (n; 2^n)$ has the expected dimension by Theorem \ref{thm:AH}. 

The above discussion shows that $ \mathcal L_2 (p+2; 2^n,   p+3-a_i \; | \; 1 \le i \le m)$ always has the expected dimension. Thus, we get 
\[
\dim [A]_{p+2} = \binom{p+4}{2} - \sum_{i=1}^m \binom{p+4 -a_i}{2} - 3n. 
\]
Using Condition \eqref{eq:exceptionalCase} again, 
it follows that 
\begin{equation*}
\begin{split}
\dim [A]_{p+2} - \dim [A]_{p-1} & = \binom{p+4}{2} -  \binom{p+1}{2}  - \sum_{i=1}^m \left [ \binom{p+4 -a_i}{2} - \binom{p+1 -a_i}{2} \right] - 3n 
\\
& = 3 p + 6 - \sum_{i=1}^m \big [3 (p-a_i) + 6 \big ] - 3n 
\\
& = 3 \left [  m + n +  \sum_{i=1}^m (p+1-a_i) - \sum_{i=1}^m (p - a_i + 2) - n \right ] 
\\
& = 0, 
\end{split}\end{equation*} 
as desired. Now the argument is complete. 
\end{proof} 

It will be convenient to rewrite Theorem \ref{xL3 general} in the following way (also using Lemma \ref{lem:rewrite}).

\begin{theorem} 
         \label{thm:mult cube}
Let  $I \subset R$ be an ideal generated by powers of  general linear forms such that $A = R/I$ is artinian.   Write $p = p(A)$ and $n_j = n_j (A)$.  
Let $\ell \in R$ be a general linear form. Then there is at least one degree $j$ where the homomorphism 
$\times \ell^3 : [A]_{j-3} \rightarrow [A]_j$ fails to have maximal rank if and only if 
\[
 n_{p+1} = p + 2 - (n_p + n_{p-1} + \cdots + n_0) \ge 4,   
\]
this number is even, and $n_{p+2} = n_{p+1}$. Moreover, in this case multiplication fails maximal rank in exactly one degree, namely if $j = p+2$. Here we have $\dim [A]_{p-1} = \dim [A]_{p+2}$,  but the cokernel has dimension 1.
\end{theorem}

For comparison with Theorem \ref{MM 5.1} and with \cite{MM} Remark 6.2 we specialize Theorem \ref{xL3 general} to the case of uniform powers.

\begin{corollary}
     \label{cor:L3 and uniform power} 
Let $R = K[x,y,z]$ and let $I = (L_1^{t}, \dots, L_s^{t})$, where the $L_i$ are general linear forms and $t \ge 1$. Let $L$ be a general linear form. Then the homomorphism $\times L^3 : [R/I]_{j-3} \rightarrow [R/I]_j$ fails to have maximal rank in all degrees if and only if $s-1$ divides $t$ and $s \ge 4$ is even.  Moreover, multiplication fails maximal rank in exactly one degree, namely when $j = s \frac{t}{s-1}$. Here we have $\dim [R/I]_{j-3} = \dim [R/I]_{j}$,  but the cokernel has dimension 1.
\end{corollary} 

\begin{proof}
We can assume that $s\geq 4$ throughout since otherwise $I$ can be assumed to be a monomial complete intersection, so the theorem of \cite{stanley} and of \cite{watanabe} gives that $R/I$ has the SLP.  We know from Corollary \ref{formula for p - uniform} that
\[
p =  \begin{cases} 
\left \lfloor \frac{s(t-1)}{s-1} \right \rfloor & \text{ if } s \le t \\
t - 1 & \text{ if } s \ge t+1. 
\end{cases}
\]
We consider the latter possibility first. Since $p = t-1$, we get from Theorem \ref{xL3 general} that $\times L^3$ fails to have maximal rank in all degrees if and only if $s = t+1$ is even; since $s \geq t+1$, this is equivalent to $s-1$ divides $t$ and $s$ is even. In this case multiplication fails to have maximal rank exactly when $j = p+2 = t+1 = s$.

We now assume that $s \leq t$. Then
\[
p = \left \lfloor \frac{s(t-1)}{s-1} \right \rfloor = \left \lfloor t + \frac{t-s}{s-1} \right \rfloor.
\]
In particular, $p \geq t$ so $t \leq p \leq p+1$ and we obtain from Theorem \ref{xL3 general} that $\times L^3$ fails to have maximal rank in all degrees $j$ if and only if
\[
s = p+2 - s(p+1-t) \geq 4
\]
is even. The above equality is equivalent to
\[
p = -2 + \frac{st}{s-1}.
\]
If $\times L^3$ fails to have maximal rank in all degrees then $s-1$ divides $t$, since $p$ is an integer and $s-1$ has no nontrivial factors in common with $s$. 

Conversely, assume that $s-1$ divides $t$ and $s$ is even. Then from the above formula for $p$ we get
\[
p = \left \lfloor \frac{st-s}{s-1} \right \rfloor = \left \lfloor \frac{st}{s-1} - 1 - \frac{1}{s-1} \right \rfloor = \frac{st}{s-1}-2
\]
so $\times L^3$ fails to have maximal rank in all degrees.
\end{proof}


\section{Applications to WLP and SLP} 
\label{sec:appl}

We will establish results on the Strong Lefschetz Property for certain ideals in three variables and on the Weak Lefschetz Property in four variables. 

We first use our Theorem \ref{xL2} as a key ingredient for proving the following result.

\begin{theorem} \label{L2 in ideal}

Let $I = (L_0^2, L_1^{a_1},\dots, L_s^{a_s}) \subset R = K[x_1,x_2,x_3]$, with $s \geq 2$, where the $L_i$ are general linear forms. Let $L$ be a general linear form. Then $\times L^d : [R/I]_{j-d} \rightarrow [R/I]_j$ has maximal rank, for all $d$ and $j$. That is, such an algebra has the Strong Lefschetz Property.

\end{theorem}

\begin{proof}
This follows immediately from  Theorem \ref{xL2} and Corollary \ref{cor:square}.
\end{proof}

Using Theorem \ref{L2 in ideal} and Proposition \ref{lem:exchange} we will generalize the following result.

\begin{theorem}[\cite{MMN2} Theorem 4.6] \label{MMN2 4.6}
Let $I = (L_0^2,L_1^{a_1},L_2^{a_2},L_3^{a_3},L_4^{a_4}) \subset R = K[x_1,x_2,x_3,x_4]$, where the $L_i$ are general linear forms. Then $R/I$ has the WLP.
\end{theorem}

%
%
%

\begin{corollary} \label{four variables}
Let $I = (L_0^2, L_1^{a_1},\dots, L_s^{a_s}) \subset R = K[x_1,x_2,x_3,x_4]$, with $s \geq 3$, where the $L_i$ are general linear forms. Then $R/I$ has the WLP.
\end{corollary}

\begin{proof}
This is a consequence of from  Theorem \ref{xL2} and  Proposition \ref{lem:exchange} (b), taking $A = k[x_1,x_2,x_3,x_4]/(L_1^{a_1},\dots,L_s^{a_s})$, $b=1$ and  $k = 2$.
\end{proof}


\begin{remark} \label{no higher power works}
In view of Theorem \ref{L2 in ideal} and Corollary \ref{four variables}, one can ask if  results along the same lines hold in more variables. 
This is not the case. 

\smallskip

(i) 
In a polynomial ring in four variables,  for the ideal $(L_0^2,L_1^4,L_2^4,L_3^4,L_4^4)$, multiplication by $L^3$ fails (by 1) to have maximal rank. Similarly, for the ideal $(L_0^2,L_1^6,L_2^6,L_3^6,L_4^6)$, multiplication by $L^2$ fails (by 1) to have maximal rank. To justify these, one can check the first example using Cremona transformations. For the second, Brian Harbourne showed us a nice geometric argument to show that maximal rank fails. 
Given six general points in $\mathbb P^3$, one has to produce an octic surface that has multiplicities $7,7,3,3,3,3$ respectively at those points. First one notes that the octic must have four linear components. Stripping these away, one needs a quartic surface with multiplicities $3,3,2,2,2,2$ respectively at those points. Now a Cremona transformation shows that the linear system of such quartics has the same dimension as the linear system of quadrics that have multiplicities $2, 2, 1,1$ at four general points. 
It follows that these systems have dimension one. As Harbourne showed us, the unique quartic surface with multiplicities $3,3,2,2,2,2$ can in fact be realized as a union of two singular quadric surfaces, although we do not need this fact here.

\smallskip

(ii) The ideal $(L_0^2, L_1^2, L_2^6,L_3^6,L_4^6,L_5^6)$ in a polynomial ring in 5 variables does not have the WLP. The argument is similar to the argument sketched in (i) above.
\end{remark}

Note that the definition of $p$ (see \eqref{eq:def p}) gives that we always have 
\[
 n_{p+1} \ge p + 2 - (n_p + n_{p-1} + \cdots + n_0).  
\]
Thus, if this inequality  is strict and $p$ does not change when adding a new generator to $I$, multiplication by a cube on the new algebra will still have maximal rank.

\begin{theorem}
     \label{thm:SLP char}
Let $A = R/I$ be an artinian algebra, where   $I$ is an ideal generated by powers  of general linear forms. Let $\ell, L \in R$ be two general linear forms. Then the following conditions are equivalent: 
\begin{itemize}
  \item[(a)] $A/\ell^3 A$ has the SLP. 
  
  \item[(b)] Multiplication by $\ell^3$ has maximal rank on $A/L^b A$ (in each degree) whenever $3 \le b \le p(A)$. 
  
  \item[(c)] Multiplication $\times \ell^3: [A/L^b A]_{j-1} \to [A/L^b A]_{j+2}$ has maximal rank whenever $3 \le b \le p(A)$ and $j = p (A/ L^b A)$. 
\end{itemize}
\end{theorem}

\begin{proof}
The equivalence of conditions (b) and (c) is a consequence of Theorem \ref{thm:mult cube}. 
\smallskip

Assume that condition (b) is satisfied. In order to show (a) fix some integer $b \ge 1$ and consider the map $\times L^b: [A/\ell^3 A]_{j+2-b} \to [A/\ell^3 A]_{j+2}$, where 
$j$ is any integer. We have to show it has maximal rank. 

This is true if $b \in \{1, 2\}$ because $A/\ell^3 A$ has the WLP by  Theorem~\ref{SS theorem},  and multiplication by $L^2$ on $A/\ell^3 A$ has maximal rank in each degree by Theorem \ref{xL2} or \cite{AA}.  

Let $b \ge 3$. Thus, we may assume $j \ge 0$. 
Since the class of $L^j$ is not zero in $[A/\ell^3 A]_j$ whenever $[A/\ell^3 A]_j\neq 0$, it follows that the map $\times L^b: [A/\ell^3 A]_{j+2-b} \to [A/\ell^3 A]_{j+2}$ has 
maximal rank if its domain or codomain is one-dimensional. Thus, we may assume that this is not true, and then Theorem~\ref{thm:mult cube} gives that both maps 
$\times \ell^3: [A]_{j-1-b} \to [A]_{j+2-b}$ and $\times \ell^3: [A]_{j-1} \to [A]_{j+2}$ have maximal rank. Combined with assumption (b), Proposition~\ref{lem:exchange} 
shows that $\times L^b: [A/\ell^3 A]_{j+2-b} \to [A/\ell^3 A]_{j+2}$ has the desired maximal rank property if $b \le p = p(A)$. 

Let $b > p$. As above our claim follows if we know that $\times \ell^3: [A/L^b A]_{j-1} \to [A/L^b A]_{j+2}$ has maximal rank. By Theorem \ref{thm:mult cube}, this is true if $j \neq p(A/L^b A)$. It remains 
 to consider the case, where $j = p(A/L^b A)$. Notice that $b > p$ implies $p (A) = p(A/L^b A)$. Thus,     
 we get $j = p(A/L^b A) = p$. If $b \ge p+2$, then $n_i (A/L^b A) = n_i (A)$ for all $i \le p+1$. Hence 
Theorem~\ref{thm:mult cube} yields that $\times \ell^3: [A/L^b A]_{j-1} \to [A/L^b A]_{j+2}$ has 
maximal rank if and only if the map $\times \ell^3: [A]_{j-1} \to [A]_{j+2}$ does. Since we have seen 
that we may assume the truth of the latter statement we are done if $b \ge p+2$. 

If $b = p+1$ we argue similarly. By definition of $p = p(A)$ we know $n_{p+1} (A) \ge p + 2 - (n_p (A) + n_{p-1}(A) + \cdots + n_0 (A))$. It follows that 
$n_{p+1} (A/L^b A) > p + 2 - (n_p (A/L^b A) + n_{p-1}A/L^b A) + \cdots + n_0 (A/L^b A))$. Since $j = p (A/L^b A) = p$ Theorem~\ref{thm:mult cube} 
shows that  $\times \ell^3: [A/L^b A]_{j-1} \to [A/L^b A]_{j+2}$ has maximal rank, as desired.  
\smallskip

Assume now $A/\ell^3 A$ has the SLP. Fix some integer $b$ with $3 \le b \le p(A)$. 
We have to show that the map $\times \ell^3: [A/L^b A]_{j-1} \to [A/L^b A]_{j+2}$ has maximal rank,  where $j = p (A/ L^b A)$. 
This is clear, if any of its domain, codomain, or cokernel are zero. Assume that this is not the case. 
Then it follows that $[A/\ell^3 A]_{j+2}$ is not trivial. We now consider two cases. 

First, assume that  $\times \ell^3: [A]_{j-1} \to [A]_{j+2}$ fails to 
have maximal rank. Then we must have (see Theorem \ref{thm:mult cube}) that  $j = p(A) = p$ and 
that $n_{p+1} (A) = p + 2 - (n_p (A) + n_{p-1} (A) + \cdots + n_0 (A))$ is even. Since $b \le p$ we 
conclude that $n_{p+1} (A/L^b A) = n_{p+1} (A) + 1$ is odd. Using Theorem \ref{thm:mult cube} 
again we get that $\times \ell^3: [A/L^b A]_{j-1} \to [A/L^b A]_{j+2}$ has maximal rank because $j = p(A/L^b A) = p$, as desired.  

Second,  assume that the map  $\times \ell^3: [A]_{j-1} \to [A]_{j+2}$ has maximal rank. Thus, it is injective
because we saw above that we may assume $[A/\ell^3 A]_{j+2} \neq 0$. If $b > j+2$, then $[A]_{j-1} = [A/L^b A]_{j-1}$ and  $[A]_{j+2} = [A/L^b A]_{j+2}$, and so 
$\times \ell^3: [A/L^b A]_{j-1} \to [A/L^b A]_{j+2}$ is injective. If $b \le j+2$, then $[A/\ell^3 A]_{j+2} \neq 0$ implies $[A/\ell^3 A]_{j+2-b} \neq 0$. Using $b \ge 3$ and Lemma~\ref{lem:p not changing}(a),  we obtain 
$j- b < j = p (A/L^b A)  \le p (A)$. Hence Theorem~\ref{thm:mult cube} gives that $\times \ell^3: [A]_{j-1-b} \to [A]_{j+2-b}$ has maximal rank. Since its cokernel is non-trivial it 
must be injective. Now Diagram~\eqref{eq:diagram} shows that $\times \ell^3: [A/L^b A]_{j-1} \to [A/L^b A]_{j+2}$ has maximal rank because the vertical map on the right 
$\times L^b: [A/\ell^3 A]_{j+2-b} \to [A/\ell^3 A]_{j+2}$  has this property by assumption. This completes the argument. 
\end{proof}

\begin{corollary}
    \label{cor:same degree}
Consider $A = R/I$, where $I$ is an ideal generated by $s \ge 2$ $t$-th powers of general linear forms, and let $\ell \in R$ be another general linear form. Then $A/\ell^3 A$ fails to have the SLP if and only if $s$ is odd and $t \ge s$. 
\end{corollary}

\begin{proof}
Set $p = p (A)$ and $n_j = n_j (A)$. In order to apply Theorem~\ref{thm:mult cube} consider  an 
integer  $b$ with $3 \le b \le p$, and set  $q = p (A/L^b A)$, and $n_j' = n_j (A/L^b A)$. By 
Lemma~\ref{lem:p not changing}(a), we know $q \le p$. Let $L$ be a general linear form. We consider several cases. 
\smallskip 

First, assume $t < s$. Then $p = t-1$ because 
$n_0 + \cdots + n_{t-1} =  0 \le t-1$ and $n_0 + \cdots + n_{t-1} + n_{t} = s \ge  t+1$, by 
assumption. Since $b \le p$ we get $n_0' + \cdots + n_{t-1}' = t -b \le t-1$, which shows that $q = p = t-1$. 
Hence, Theorem~\ref{thm:mult cube} gives that multiplication by $\ell^3$ on $A/L^b A$ has maximal 
rank, and so $A/\ell^3 A$ has the SLP by Theorem~\ref{thm:SLP char}. 
\smallskip

Second, assume $s$ is even. Since $n_{q+1}'$ must be positive, we have $n_{q+1}' \in \{1, s+1\}$. In particular, $n_{q+1}'$ is odd, and so multiplication by $\ell^3$ on $A/L^b A$ has maximal rank by Theorem~\ref{thm:mult cube}, which implies that $A/\ell^3 A$ has the SLP. 
\smallskip

Third, assume $t \ge s$ and $s$ is odd. We will show that multiplication by $\ell^3$ on $A/L^b A$ fails to have maximal rank for $b = s$. 

The fact that $n_0 + \cdots  + n_{t} = s$ implies that $p \ge t$. Thus, we get 
$3 \le s = b \le t \le p$, and so $s$ is in the range of integers $b$ considered in Theorem~\ref{thm:SLP char}. 

We now claim that $q = t-1$. Indeed, we have $n_0' + \cdots + n_{t-1}' = t - s \le t-1$ and 
\[
n_0' + \cdots + n_{t-1}' + n_{t}' = (t-s) + n_t' = (t-s) + (s+1) = t+1,
\]
which proves the claim, and 
$s+1 = n_{q+1}' = q+2 - (n_0' + \cdots + n_{q}')$.  Hence, multiplication by $\ell^3$ on $A/L^s A$ fails to have maximal rank, and so $A/\ell^3 A$ does not have the SLP by Theorem~\ref{thm:SLP char}. In fact, this failure occurs by one (failure of surjectivity) in exactly one degree, namely $j= q+2 = t+1$.
\end{proof}

For $A$ as in Corollary \ref{cor:same degree}, it is also interesting to understand the set of $b$ for which $\times \ell^b$ fails to have maximal rank on $A/\ell^3A$. Recall that for a graded artinian algebra $B$, the {\em Castelnuovo-Mumford regularity} $\reg (B)$ is the degree of the last non-zero component of $B$.

\begin{corollary} \label{which b fail}
Let $A = R/I$, where $I$ is an ideal generated by $s \geq 2$ $t$-th powers of general linear forms, and let $\ell , L \in R$ be  general linear forms. Assume that $s$ is odd and that $t \geq s$. Assume also that $b \leq t$. Then there is at least one degree in which $\times L^b$ fails to have maximal rank on $A/\ell^3A$ if and only if $s$ divides $b$ and $b < \reg(A/\ell^3A)$. Moreover, in this situation there is exactly one such degree, and multiplication by $L^b$ fails both injectivity and surjectivity by one there.
\end{corollary}

\begin{proof}
The condition on the regularity index is because if $b \geq \reg(A/\ell^3A)$ then $\times L^b$ necessarily has maximal rank in all degrees for general $L$. From now on we assume that this condition holds.

By Corollary \ref{cor:same degree}, $A/\ell^3A$ does not have the SLP.  It follows from the commutative diagram (\ref{eq:diagram}) (taking $k=3$) that $\times \ell^3$ is surjective on $A/L^bA$ from degree $j$ to degree $j+b$ if and only if $\times L^b$ is surjective on $A/\ell^3A$ in the same degree. Furthermore, when surjectivity fails, the cokernels are the same. It follows from Proposition \ref{lem:exchange} (b) that if $\times \ell^3$ is injective on $A/L^bA$ in any degree then then $\times L^b$ is injective on $A/\ell_3A$ in the same degree. Note also that the failure of maximal rank in Theorem \ref{thm:mult cube}  is a failure of both injectivity and surjectivity by one, and that there is exactly one degree in which it occurs.

For this proof we deviate slightly from the notation of Corollary \ref{cor:same degree} and we set $p = p(A/L^b A)$.  We have already seen the result when $b = s$.

Note that
\[
n_0 = \dots = n_{b-1} = 0 \ \ \hbox{ and } \ \ n_t = n_{t+1} = \dots = s+1.
\]
Furthermore,
\[
\hbox{if $b < t$ then } n_b = \dots = n_{t-1} = 1.
\]
We obtain
\[
n_0 + \dots + n_{t-1} = t-b
\]
and 
\[
n_0 + \dots + n_{t+j} = t-b + (j+1)(s+1)
\]
for $j \geq 0$. To find $p$ we need the largest $j$ so that 
\[
t-b + (j+1)(s+1) \leq t+j,
\]
i.e. such that
\[
(j+1) s \leq b-1.
\]
Thus
\[
p = t + \alpha - 1, \ \ \hbox{ where } \ \ \alpha = \max \{ j \ | \ js \leq b-1 \}
\]
and 
\[
n_0 + \dots + n_p = t-b + \alpha (s+1).
\]
Note that in any case we have $\alpha \geq 0$, so $p \geq t-1$.
%
%
%
%
We compute
\begin{equation} \label{p+2eqn}
p+2-(n_0+\dots+n_p) = t+\alpha + 1 - (t-b + \alpha(s+1) = b - \alpha s + 1.
\end{equation}
We consider two cases.  If $b = ks$ for some $k \geq 1$ then $\alpha = k-1$ and
\[
p+2-(n_0+\dots + n_p) = ks - (k-1)s+1 = s+1 = n_{p+1}.
\]
Since the other conditions of Theorem \ref{thm:mult cube} are clearly satisfied, we conclude that there is exactly one degree in which $\times L^b$ fails to have maximal rank on $A/\ell^3 A$, and the dimensions of the kernel and of the cokernel are both 1.

On the other hand, assume $b \neq ks$ for any positive integer $k$. Then from (\ref{p+2eqn}) we see that $p+2-(n_0+\dots + n_p)$ will not equal $s+1 = n_{p+1}$, so $\times \ell^3$ has maximal rank on $A/L^b A$ in each degree. In particular, whenever $\times \ell^3$ is surjective on $A/L^bA$, we get $\times L^b$ is surjective on $A/\ell^3A$ in the same degree. 
Suppose instead that $\times \ell^3$ is injective on $A/L^bA$ from some degree $j-3+b$ to degree $j+b$, with a positive dimensional cokernel (see (\ref{eq:diagram})). If furthermore $\times \ell^3$ has maximal rank on $A$ in all degrees then we are done by Proposition \ref{lem:exchange}. If instead there is some degree in which $\times \ell^3$ fails to have maximal rank on $A$, this degree is unique and the cokernel is one-dimensional. Then as observed in Theorem \ref{thm:SLP char}, $\times L^b$ has maximal rank on $A/\ell^3A$ in all degrees.
\end{proof}

Note that Corollary \ref{which b fail} is not exhaustive because of the condition that $b \leq t$, and we believe that the assertion holds without this condition. For example, one can check that if $s=5$ and $t = 14$, then  the failure of maximal rank for $\times L^{b}$ (in each case failing by one to be an isomorphism) occurs only for $b = 5, 10$ and 15. In particular, for $b=15$ one has $\reg(A/\ell^3 A) = 17$ and the failure comes between degrees 2 and 17.

\begin{conjecture}
Corollary \ref{which b fail} holds without the assumption that $b \leq t$.
\end{conjecture}

%
%
%

It was shown in Corollary \ref{four variables} above that if $I = (L_0^2,L_1^{a_2},\dots,L_s^{a_s}) \subset k[x_1,x_2,x_3,x_4]$, where the $L_i$ are general linear forms, then $R/I$ has the WLP. We use the results of this section to give a partial extension of this result to the case where the ideal has a cube rather than a square, and all the $a_i$ are equal.

%
%

\begin{corollary} 
      \label{for:uniform powers in 4 var}
Let $R = k[x_1,x_2,x_3,x_4]$, and $L$ a general linear form.
Let $s \geq 4$ and $t \geq 3$ be positive integers. Let $L_0,\dots,L_{s}$ be general linear forms in $R$. Let $I = (L_0^3,L_1^t,\dots,L_{s}^t)$. Then $R/I$ has the WLP if and only if one of the following conditions holds:

\begin{itemize}
\item[(i)] $s$ is odd; 
\item[(ii)] $s$ is even and $t$ is not a multiple of $s-1$.
\end{itemize}

\noindent When $R/I$ fails WLP, it does so as a failure of surjectivity by one, in exactly one degree, namely $\frac{st}{s-1}$.
\end{corollary}

\begin{proof}
In this proof, if $\deg f = b$, we will use the term {\em $f$ has  maximal rank in degree $j$} to mean that the map defined by $f$ from degree $j-b$ to degree $j$ has maximal rank. Let $A = R/(L_0^3,L_1^t,\dots,L_{s-1}^t)$.
Consider the commutative diagram
\begin{align} 
   \label{eq:diagram2}
\begin{CD}
[A]_{j-t-1} @>\phi_1 = L_{s}^t>> [A]_{j-1} @>>> [A/L_s^tA]_{j-1} @>>> 0 \\
@ VV{\psi_1 =  L}V   @ VV{\psi_2 =  L}V  @ VV{\psi_3 = L}V  \\
[A]_{j-t} @>\phi_2 = L_{s}^t>> [A]_{j} @>>> [A/L_s^tA]_{j} @>>> 0 \\
 @ VVV @  VVV @  VVV \\
 [A/LA]_{j-t} @>\phi_3 = \bar L_{s}^t >> [A/LA]_{j} @>>> [A/(L,L_s^t)A]_{j} @>>> 0 \\
  @ VVV @  VVV @  VVV \\
  0 && 0 && 0. 
  \end{CD}
\end{align}
Writing $S = R/(L) \cong k[x_1,x_2,x_3]$, we have $A/LA$ is isomorphic to a quotient of $S$ modulo the ideal generated by the restrictions $\bar L_0^3,\bar L_1^t,\dots, \bar L_{s-1}^t$.


Our focus is to determine when $\psi_3$ has maximal rank. Part of our argument will be to study $\phi_3$. We note that in this case, in the notation of Corollary \ref{which b fail}, we have $b=t$ so Corollary \ref{which b fail} applies. 

If $s=4$ then $(L_0^3,L_1^t,L_2^t,L_3^t)$ is a complete intersection, so $\phi_1, \phi_2, \psi_1$ and $\psi_2$ all have maximal rank. The only non-trivial situation is when all four of these maps are injective, so assume that this is the case. In this situation, maximal rank of $\psi_2$ is equivalent to that of $\phi_3$ (for instance by applying the Snake Lemma using the first two rows and using the first two columns). Thus by Corollary \ref{cor:same degree} and Corollary \ref{which b fail}, in this situation $\psi_3$ fails maximal rank if and only if $t$ is a multiple of $3$, and the failure is  by one for $\phi_3$ to be an isomorphism. In terms of the degree, we know that $\psi_3$ has the same cokernel as $\phi_3 = \times L_4^t : [S/(L_0^3,L_1^t,L_2^t,L_3^t)]_{j-t} \rightarrow [S/(L_0^3,L_1^t,L_2^t,L_3^t)]_j$, which in turn has the same cokernel as $\times L_0^3 : [S/(L_1^t,L_2^t,L_3^t,L_4^t)]_{j-3} \rightarrow [S/(L_1^t,L_2^t,L_3^t,L_4^t)]_j$. By Corollary \ref{cor:L3 and uniform power} , this is one-dimensional when $j = 
\frac{4t}{3}$.

As a consequence, we have:  

\begin{quotation}
{\em For $s=4$, $R/(L_0^3,L_1^t,L_2^t,L_3^t,L_4^t)$ has the WLP if and only if $t$ is not a multiple of 3. If $t$ is a multiple of 3 then there is exactly one degree where $\psi_3$ fails maximal rank, namely $\frac{4t}{3}$, and it is a failure by one of surjectivity.}
\end{quotation}

For $s \geq 5$, we can no longer expect $\phi_1$ and $\phi_2$ to have maximal rank in all degrees.

Let $s = 5$. Then $s-1 = 4$ and $\phi_3$ has maximal rank in all degrees, by Corollary \ref{cor:same degree}. When $\phi_3$ is surjective then so is $\psi_3$, so we will assume that $\phi_3$ is injective and not surjective. 

If $t$ is not a multiple of 3 then $\psi_1$ and $\psi_2$ have maximal rank in all degrees, by the last case. Then by Proposition \ref{lem:exchange} (b) (taking $b=1$), $\psi_3$ has maximal rank. 

Now assume that $t$ is a multiple of 3. If $\psi_1$ is surjective then so is $\psi_2$ (since it is multiplication on a standard graded algebra), hence also $\psi_3$. If $\psi_2$ is surjective then so is $\psi_3$. So we will assume that neither $\psi_1$ nor $\psi_2$ is surjective. 
Both $\psi_1$ and $\psi_2$ have the property that each has maximal rank in all but one degree (hence they do not fail in the same degree). 
If neither $j$ nor $j-t$ is equal to this degree, 
the same argument holds as before, giving maximal rank for $\psi_3$. If $\psi_1$ fails maximal rank (surjectivity) by one then $\psi_2$ is surjective, hence so is $\psi_3$. It remains to check the case where $\psi_2$ fails surjectivity by one.

So assume that $\psi_2$ fails to have maximal rank, and let $j$ the the unique degree where this failure occurs.  We are assuming that $\phi_3$ is injective but not surjective, so the bottom row of (\ref{eq:diagram2}) is a short exact sequence. 
By the case $s=4$, $\dim [S/(\bar L_0^3,\bar L_1^t,\dots,\bar L_4^t)]_{j} = 1$. This forces $\dim [S/(\bar L_0^3,\bar L_1^t,\dots,\bar L_4^t)]_{j-t} > 0$ since the quotient is a standard graded algebra, so by injectivity the dimension must be 1 and hence $\phi_3$ is also surjective, and we have a contradiction.

We conclude:

\begin{quotation}
{\em For $s=5$, $R/(L_0^3,L_1^t,L_2^t,L_3^t,L_4^t, L_5^t)$ has the WLP for all $t$.}
\end{quotation}

Since the case $s=4$ was special, we also consider the case $s=6$. Now we can assume that $\psi_1$ and $\psi_2$ have maximal rank in all degrees  thanks to the just-completed case $s=5$. 
We know from Corollary \ref{which b fail} that $\phi_3$ has maximal rank except when $t$ is a multiple of 5. Furthermore, reasoning analogously to the case $s=4$, failure of maximal rank occurs only when $j= \frac{6t}{5}$. In this latter case the kernel and cokernel of $\phi_3$ are both one-dimensional. 

If $t$ is not a multiple of 5, or if $t$ is a multiple of 5 but $j \neq \frac{6t}{5}$, then $\phi_3$ has maximal rank, and so by the argument of Proposition \ref{lem:exchange} (b), $\psi_3$ has maximal rank. So it remains to consider the case that $t$ is a multiple of 5 and $j = \frac{6t}{5}$. We can further assume without loss of generality that $\psi_1$ and $\psi_2$ are injective, since as before surjectivity trivially gives the desired surjectivity for $\psi_3$.

In the commutative diagram (\ref{eq:diagram2}), our assumptions give us short exact sequences in the first two columns. Then from the Snake Lemma we have a long exact sequence
\[
\begin{array}{ll}
0 \rightarrow \ker \phi_1 \rightarrow \ker \phi_2 \rightarrow  \ker \phi_3 \rightarrow [R/(L_0^3,L_1^t,\dots,L_{s}^t)]_{j-1} \stackrel{\psi_3}{\longrightarrow} [R/(L_0^3,L_1^t,\dots,L_{s}^t)]_{j} \\ \\
\hspace{1in} \rightarrow [R/(L_0^3,L_1^t,\dots,L_{s}^t,L)]_{j-1} \rightarrow 0
\end{array}
\]
(where $s=6$).
Since $\psi_2$ is injective, we have that multiplication by $L$ is injective in all degrees up to $j = \frac{6t}{5}$, so the compositions $\phi_1$ and $\phi_2$ are injective for $j = \frac{6t}{5}$. Then $\ker \psi_3 \cong \ker \phi_3$ is one-dimensional, and the cokernel is also one-dimensional in degree $\frac{6t}{5}$ as claimed. We conclude:

\begin{quotation}
{\em For $s=6$, $R/(L_0^3,L_1^t,L_2^t,L_3^t,L_4^t,L_5^t,L_6^t)$ has the WLP if and only if $t$ is not a multiple of 5. If $t$ is a multiple of 5 then there is exactly one degree where $\psi_3$ fails maximal rank, namely $\frac{6t}{5}$, and it is a failure by one of surjectivity.}
\end{quotation}

The proof now goes by induction on $s$. When $s$ is odd, the argument mimics that for  $s=5$, and when $s$ is even it mimics the argument for $s=6$.
\end{proof}


\end{document}